\documentclass[11pt]{amsart}
\usepackage[usenames]{color}
\usepackage{fullpage}
\usepackage{amscd}
\usepackage{amssymb, latexsym}
\usepackage[all,2cell,ps]{xy}
\usepackage{mathdots}

\theoremstyle{plain}
\newtheorem{thm}{Theorem}[section]
\newtheorem{theorem}[thm]{Theorem}
\newtheorem{mthm}[thm]{Main Theorem}
\newtheorem{cor}[thm]{Corollary}

\newtheorem{lemma}[thm]{Lemma}

\newtheorem{proposition}[thm]{Proposition}

\theoremstyle{definition}
\newtheorem{de}[thm]{Definition}

\newtheorem{rem}[thm]{Remark}
\newtheorem{remark}[thm]{Remark}
\newtheorem{convention}[thm]{Convention}
\newtheorem{example}[thm]{Example}

\newcommand{\Z}{\mathbb{Z}}
\newcommand{\N}{\mathbb{N}}

\newcommand{\aut}[1]{\mathrm{Aut}(#1)}

\numberwithin{equation}{section}

\begin{document}
\title{Indecomposable involutive solutions of the Yang-Baxter equation of multipermutational level 2 with abelian permutation group}

\author{P\v remysl Jedli\v cka}
\author{Agata Pilitowska}
\author{Anna Zamojska-Dzienio}

\address{(P.J.) Department of Mathematics, Faculty of Engineering, Czech University of Life Sciences, Kam\'yck\'a 129, 16521 Praha 6, Czech Republic}
\address{(A.P., A.Z.) Faculty of Mathematics and Information Science, Warsaw University of Technology, Koszykowa 75, 00-662 Warsaw, Poland}

\email{(P.J.) jedlickap@tf.czu.cz}
\email{(A.P.) A.Pilitowska@mini.pw.edu.pl}
\email{(A.Z.) A.Zamojska@mini.pw.edu.pl}


\date{\today}

\begin{abstract}
We present a construction of all finite indecomposable involutive solutions of the Yang-Baxter equation of multipermutational level at most 2 with abelian permutation group. As a consequence, we obtain a formula for the number of such solutions with a fixed number of elements. We also describe some properties of the automorphism groups in this case - in particular, we show they are regular abelian groups.
\end{abstract}

\maketitle
\section{Introduction}\label{sec:intr}

The Yang-Baxter equation is a fundamental equation occurring in integrable models in statistical mechanics and quantum field theory~\cite{Jimbo}.
Let $V$ be a vector space. A {\em solution of the Yang--Baxter equation} is a linear mapping $r:V\otimes V\to V\otimes V$ such that
\[
(id\otimes r) (r\otimes id) (id\otimes r)=(r\otimes id) (id\otimes r) (r\otimes id).
\]
Description of all possible solutions seems to be extremely difficult and therefore
there were some simplifications introduced (see e.g. \cite{Dr90}).

Let $X$ be a basis of the space $V$ and let $\sigma:X^2\to X$ and $\tau: X^2\to X$ be two mappings.
We say that $(X,\sigma,\tau)$ is a {\em set-theoretic solution of the Yang--Baxter equation} if
the mapping $x\otimes y \mapsto \sigma(x,y)\otimes \tau(x,y)$ extends to a solution of the Yang--Baxter
equation. It means that $r\colon X^2\to X^2$, where $r=(\sigma,\tau)$ satisfies the \emph{braid relation}:

\begin{equation}\label{eq:braid}
(id\times r)(r\times id)(id\times r)=(r\times id)(id\times r)(r\times id).
\end{equation}

A solution is called {\em non-degenerate} if the mappings $\sigma_x:=\sigma(x,\_)$ and $\tau_y:=\tau(\_\,,y)$ are bijections,
for all $x,y\in X$.
A solution $(X,\sigma,\tau)$ is {\em involutive} if $r^2=\mathrm{id}_{X^2}$. In the involutive case, the operation $\tau$ can be expressed by means of the operation $\sigma$ (see Theorem \ref{th:rclq}).
\begin{convention}
All solutions, we study in this paper, are set-theoretic, non-degenerate and involutive so we will call them simply \emph{solutions}.
\end{convention}

\emph{The permutation group} $\mathcal{G}(X)=\left\langle \sigma_x\colon x\in X\right\rangle$ of a solution $(X,\sigma,\tau)$ is the subgroup of the symmetric group $S(X)$ generated by mappings $\sigma_x$, with $x\in X$. The group $\mathcal{G}(X)$ is also called \emph{the involutive Yang-Baxter group} (IYB group) associated to the solution $(X,\sigma,\tau)$. A solution $(X,\sigma,\tau)$ is \emph{indecomposable} if the permutation group  $\mathcal{G}(X)$ acts transitively on $X$.

In \cite[Section 3.2]{ESS} Etingof, Schedler and Soloviev introduced, for each solution $(X,\sigma,\tau)$, the equivalence relation $\sim$ on the set $X$: for each $x,y\in X$

\[
x\sim y\quad \Leftrightarrow\quad \sigma_x=\sigma_y.
\]
They showed that the quotient set $X/\mathord{\sim}$ can be again endowed
with a structure of a solution. They call such a solution the {\em retraction} of the solution~$(X,\sigma,\tau)$ and denote it by
$\mathrm{Ret}(X)$. One can also define \emph{iterated retraction} in the following way: ${\rm Ret}^0(X,\sigma,\tau):=(X,\sigma,\tau)$ and
${\rm Ret}^k(X,\sigma,\tau):={\rm Ret}({\rm Ret}^{k-1}(X,\sigma,\tau))$, for any natural number $k>1$.
A solution $(X,\sigma,\tau)$ is called a \emph{multipermutation solution of level $m$} if $m$ is the least nonnegative integer such that
$|{\rm Ret}^{m}(X,\sigma,\tau)|=1.$ In such case, we will also say that a solution is \emph{of a multipermutational level $m$}.

Special properties of multipermutation solutions and their impact on the associated algebraic objects were investigated by various authors. The particular case of multipermutation solutions of level $2$ was first studied by Gateva-Ivanova and Majid in \cite{GIM07,GIM11} and then continued by Gateva-Ivanova in \cite{GI18}. In \cite[Section 3.3]{ESS} some examples of enumeration of indecomposable multipermutation solutions of level $2$ for small sets were given. The authors based them on classification of transitive actions of a free abelian group on sets $\mathbb{Z}_n=\mathbb{Z}/n\mathbb{Z}$. The significance of such solutions comes from the fact that each multipermutation solution of level $2$ is a \emph{generalized twisted union} of indecomposable multipermutation solutions of level $1$ or $2$ (see \cite[Section 3.4]{ESS}). In general, the indecomposable solutions play a role of bricks in constructing other solutions so recently this topic is of increasing interest (e.g. \cite[Problem 4]{Vendr}, \cite{CCP,CPR,CJO,Rump}). The main aim of our paper is to present a direct construction of finite indecomposable solutions of multipermutational level at most $2$ with abelian permutation group. This is natural restriction since Ced\'o, Jespers and Okni\'nski in \cite[Theorem 6.5]{CJO14} showed that each finite solution with an abelian permutation group is a multipermutation solution. Moreover, in \cite[Theorem 7.11]{JPZ19} we showed that each abelian group is a permutation group for a multipermutation solution of level $2$. Nevertheless, we provide an example of an indecomposable solution of multipermutational level $2$ with non-abelian permutation group (Example~\ref{ex:non-abelian}). An example of an indecomposable solution with cyclic permutation group of multipermutational level $3$ has been given recently in \cite[Example 3]{CPR}.

Here we summarize previous work on indecomposable multipermutation solutions. In \cite{ESS}, with the use of a computer program, the number of all indecomposable solutions with at most $8$ elements were found. It was also shown that (up to isomorphism) there is only one indecomposable solution with a prime number of elements (this is a cyclic permutation solutions, i.e. a multipermutation solution of level $1$ with $\sigma_x$ being a cycle). This result was recently used in \cite{CJO} for description of all finite \emph{primitive} solutions (solutions with primitive permutation group). Obviously, they are indecomposable and, by \cite[Theorem 3.1]{CJO}, they have a prime number of elements. In general, a permutation solution is indecomposable if and only if it is cyclic. In \cite{CPR} a complete classification of the indecomposable solutions of cardinality $pq$ (with $p$, $q$ not necessarily distinct prime numbers) and an abelian permutation group was given. They are of multipermutational level at most $2$. We refer to these results in Example \ref{exm:1}.

The paper is organized as follows. In Section \ref{sec:Pre} we recall basic notions and results. Section \ref{sec2} contains a construction of
solutions. In Section \ref{sec2b} we find the
main result -- the representation theorem for indecomposable solutions of multipermutational level at most $2$ \emph{of any finite cardinality} with an abelian permutation group (Main Theorem \ref{cor:sol}). This allows us to obtain a complete formula for the number of solutions in this class (\ref{thm:7}). We study also the special case when the permutation group is cyclic. In Section \ref{ssec3} we give some results on the automorphism groups for finite indecomposable solutions of multipermutational level at most~$2$ with abelian permutation group. We finish the paper with a discussion on the infinite case in Section \ref{ssec:4}.

\section{Preliminaries}\label{sec:Pre}

In \cite{Rump05} Rump introduced \emph{cycle-sets} and showed that there is a one-to-one correspondence between solutions of the Yang-Baxter equation and non-degenerate cycle sets.
\begin{de}
An algebra $(X,\alpha)$ with $\alpha\colon X^2\to X$ is a \emph{cycle-set} if the mappings $\alpha_x:=\alpha(x,\_)$, for $x\in X$,  are bijections and the inverse mappings $\alpha^{-1}_x$ satisfy the following condition for every $a,b\in X$,
\begin{equation}\label{eq:cc}
\alpha^{-1}_{\alpha^{-1}_a(b)}\alpha^{-1}_a=\alpha^{-1}_{\alpha^{-1}_b(a)}\alpha^{-1}_b.
\end{equation}
It is \emph{non-degenerate}, if the mapping
\begin{equation}\label{eq:ND}
 T\colon X\to X;\quad  a\mapsto \alpha^{-1}_a(a),
  \end{equation}
  is a bijection.
\end{de}

\begin{theorem} \cite[Proposition 1]{Rump05}\label{th:rclq}
$(X,\sigma,\tau)$ is a solution of the Yang-Baxter equation if and only if $(X,\sigma)$ is a non-degenerate cycle set. The operation $\tau$ is defined then by
\begin{align}\label{eq:ccsigma}
\tau(x,y)=\tau_y(x)=\sigma^{-1}_{\sigma_x(y)}(x),
\end{align}
for $x,y\in X$.
\end{theorem}

Rump also showed in \cite[Theorem 2]{Rump05} that each finite cycle-set is non-degenerate (for a short direct proof see \cite[Proposition 4.7]{JPZ}).
\vskip 5mm

By results of Gateva-Ivanova, multipermutation solutions of level $2$ can be characterized in an easy way.
\begin{theorem}\cite[Proposition 4.7]{GI18}\label{GI:mper}
Let $(X,\sigma,\tau)$ be a solution and $|X|\geq 2$. Then
$(X,\sigma,\tau)$ is a multipermutation solution of level $2$ if and only if the following condition holds for every $x,y,z\in X$:
\begin{align}\label{eq:2per}
\sigma_{\sigma_y(x)}=\sigma_{\sigma_z(x)}.
\end{align}

\end{theorem}

\begin{de}
A solution $(X,\sigma,\tau)$ is called \emph{$2$-reductive}, if for every $x,y\in X$
\begin{align}\label{eq:red}
\sigma_{\sigma_x(y)}=\sigma_y.
\end{align}
\end{de}

It was shown in~\cite{JPZ19} that $2$-reductive solutions can be used as ground stones for constructing multipermutation solutions of level~$2$.

\begin{theorem}\cite[Theorem 4.11]{JPZ19}\label{th:pi-isotope}
Let $(X,\sigma,\tau)$ be a $2$-reductive solution and let~$\pi$ be a permutation of~$X$ that
satisfies
\[\sigma_{\pi(y)}\pi\sigma_x=\sigma_{\pi(x)}\pi\sigma_y,\]
for all $x,y\in X$.
Then $(X,\sigma',\tau')$, where $\sigma'_x=\sigma_x \pi$ and $\tau'_y=\pi^{-1}\tau_{\pi(y)}$ is a solution of multipermutational level at most~$2$.
\end{theorem}

\begin{example}\label{ex:non-abelian}
 Let $X=\Z_n\times\{0,1\}$, for some $n\in\N$, and let $\sigma_{(a,i)}((b,j))=(b+i,j)$.
 This is a $2$-reductive solution, according to
 \cite[Theorem 3.6]{JPZ19}: it is the sum of a trivial affine mesh $(\Z_n,\Z_n,\left(\begin{smallmatrix} 0 & 0
 \\ 1 & 1 \end{smallmatrix}\right))$.
 
 Let now $\pi$ be a permutation of~$X$ defined by $\pi((a,i))=(-a,1-i)$. We have
 \[
 \sigma_{\pi((a,i))}\pi\sigma_{(b,j)}(c,k)=
 \sigma_{(-a,1-i)}\pi (c+j,k)
 =\sigma_{(-a,1-i)} (-c-j,1-k)=
 (-c-j+1-i,1-k)
 \]
 and this expression is symmetric to the exchange $(a,i)\leftrightarrow(b,j)$.
 According to Theorem~\ref{th:pi-isotope}, the solution $(X,\sigma',\tau')$ is a multipermutation
 solution of level~$2$.
 Since $\sigma_{(a,0)}'=\pi$ and $\sigma_{(a,1)}'=\sigma_{(0,1)}\pi$, for all $a\in\Z_n$,
 the permutation group
 of the solution is generated by $\sigma_{(0,1)}$ and $\pi$ and it is transitive. Moreover, we have $\sigma_{(0,1)}^{-1}=\pi^{-1}\sigma_{(0,1)}\pi$ and therefore 
 the permutation group is isomorphic to $\Z_n\ltimes \Z_2$.
\end{example}

Actually, every solution of multipermutational level~$2$ can be obtained this way from a~2-reductive solution.

\begin{theorem}\cite[Theorem 7.12]{JPZ19} \label{th:2persol}
Let $(X,\sigma,\tau)$ be a multipermutation solution of level~$2$
and $e\in X$. Then $(X,L,\mathbf{R})$, where $L_x=\sigma_x\sigma_e^{-1}$ and $\mathbf{R}_y=\sigma_e\tau_{\sigma_e^{-1}(y)}$, for $x,y\in X$, is a $2$-reductive solution.
\end{theorem}
We call the solution $(X,L,\mathbf{R})$ from Theorem \ref{th:2persol} the $\sigma_e^{-1}$-\emph{isotope} of $(X,\sigma,\tau)$. Clearly, $L_e=id$. Moreover, by \cite[Theorem 6.12]{JPZ19}, the solution $(X,L,\mathbf{R})$ satisfies the condition from Theorem~\ref{th:pi-isotope}:

\begin{align}
& L_{\sigma_e(x)}\sigma_e L_y=L_{\sigma_e(y)}\sigma_e L_x. \label{eq:sigma}
\end{align}

A mapping $\varphi\colon X\to X'$ is a \emph{homomorphism} of two solutions $(X,\sigma,\tau)$ and $(X',\sigma',\tau')$ if, for each $x\in X$,
\[
\varphi\sigma_x=\sigma'_{\varphi(x)}\varphi.
\]

If a finite abelian group $G$ has a decomposition as a direct sum $G=\sum_{i=1}^{t}C_{i}$, where $C_i$ is a cyclic group of order $n_i$ and $n_1|n_2$, $n_2|n_3$, $\ldots$, $n_{t-1}|n_t$, then one says that $G$ has \emph{invariant factors} $(n_1,n_2,\ldots,n_t)$ see e.g. \cite[Chapter 6]{R}.
\begin{theorem}\cite[Corollary 6.14]{R}\label{thm:Rotman}
Two finite abelian groups $G$ and $H$ are isomorphic if and only if they have the same invariant factors.
\end{theorem}

\section{A construction of solutions}\label{sec2}

In this section we present a construction of a
solution with an abelian permutation group.
Recall that $\Z_k$ is a shortcut for $\Z/k\Z$; in particular $\Z_1$ is a trivial group.

\begin{theorem}\label{thm:1}
 Let $n_1,n_2\in \Z^+$ be such that
$n_1\mid n_2$. Let $r\in\{0,1,\ldots,n_2/n_1-1\}$ be such that $n_2\mid n_1r^2$. Then
 $(X,\sigma,\tau)$ with~$X=\Z_{n_1}\times \Z_{n_2}$ and
 \begin{align}\label{eq:indec}
 \sigma_{(a,i)}((b,j))=(b-ar+i,j+ir-ar^2+1)
 \end{align}
 is an indecomposable solution of size~$n_1n_2$
 and multipermutational level at most~$2$
 with the permutation group $\mathcal{G}(X)$ isomorphic to $\Z_{n_1}\times \Z_{n_2}$.
\end{theorem}

\begin{proof}
First of all, we prove that the expression is
well defined since we mix elements from $\Z_{n_1}$
and $\Z_{n_2}$ in both coordinates.
The addition of~$i$ in the first coordinate is well defined
since $n_1\mid n_2$. The substraction of~$ar^2$
in the second coordinate is well defined as well
since $n_1r^2\equiv 0\pmod {n_2}$.

In the beginning, we check directly that $(X,\sigma)$ is
a cycle-set. Clearly each mapping $\sigma_{(a,i)}$ is bijective and
\[
\sigma_{(a,i)}^{-1}( (b,j))=(b+ar-i,j-ir+ar^2-1).
\]

Now the left side of \eqref{eq:cc} is the following:
 \begin{multline*}
  \sigma_{\sigma_{(a,i)}^{-1}((b,j))}^{-1}\sigma_{(a,i)}^{-1}((c,k))=\sigma_{(b+ar-i,j-ir+ar^2-1)}^{-1}((c+ar-i,k-ir+ar^2-1))=\\
  (c+ar-i+(b+ar-i)r-j+ir-ar^2+1,k-ir+ar^2-1
  -(j-ir+ar^2-1)r+(b+ar-i)r^2-1)\\
  =(c+ar-i+br-j+1,k-ir+ar^2-jr-r+br^2-2).
 \end{multline*}

 This expression is symmetric with respect to the exchange $a\leftrightarrow b$ and $i\leftrightarrow j$ and hence $\sigma$ satisfies the cycle condition \eqref{eq:cc}. Since $X$ is finite, we obtain that $(X,\sigma)$ is a non-degenerate cycle-set and, by Theorem \ref{th:rclq}, $(X,\sigma,\tau)$ is a solution with $\tau_{(b,j)}((a,i))=(a+br-(j+1),i-(j+1)r+br^2-1)$.

Furthermore,
 \begin{multline*}
  \sigma_{(a,i)}\sigma_{(b,j)} ((c,k)) =\sigma_{(a,i)}
  ((c-br+j,k+jr-br^2+1))=(c-br+j-ar+i,\\
  k+jr-br^2+1+ir-ar^2+1)
  =\sigma_{(b,j)} ((c-ar+i,k+ir-ar^2+1))=\sigma_{(b,j)}\sigma_{(a,i)} ((c,k))
 \end{multline*}
 which shows that the permutation group $\mathcal{G}(X)$ is abelian.
Moreover,
 \begin{multline*}
  \sigma_{\sigma_{(a,i)}((b,j))}((c,k)) =
  \sigma_{(b-ar+i,j+ir-ar^2+1)} ((c,k))
  =(c-(b-ar+i)r+j+ir-ar^2+1,\\
  k+(j+ir-ar^2+1)r-(b-ar+i)r^2+1)
  =(c-br+j+1,k+jr+r-br^2+1).
 \end{multline*}
 The end of above expression does not depend on $(a,i)$, proving that the solution $(X,\sigma,\tau)$ is of multipermutational level at most~$2$, by Theorem~\ref{GI:mper}.

 Now, the permutation
 $\sigma_{(0,0)}$ is evidently of order~$n_2$.
 The permutation $\sigma_{(0,1)}\sigma_{(0,0)}^{-r-1}$
 sends $(a,i)$ to $(a+1,i)$ and therefore it is
 a permutation of order $n_1$. Furthermore,
 \[ \sigma_{(0,0)}^j \ \big(\sigma_{(0,1)}\sigma_{(0,0)}^{-r-1}\big)^b \ ((a,i)) = (a+b,i+j)\]
 and therefore $\mathcal{G}(X)$ acts transitively on~$X$. A transitive action of an abelian group has to be regular
 and therefore the order of  $\mathcal{G}(X)$ is~$n_1n_2$.
 Finally,
 $\langle\sigma_{(0,0)}\rangle\cap\langle\sigma_{(0,1)}\sigma_{(0,0)}^{-r-1}\rangle=\{\mathrm{id}\}$
 and therefore the permutation group $\mathcal{G}(X)$ is isomorphic
 to $\Z_{n_1}\times \Z_{n_2}$.
\end{proof}

We will denote the solution described in Theorem \ref{thm:1} by $\mathcal{C}(n_1,n_2,r)$, for $n_1,n_2,r$ specified in Theorem \ref{thm:1}.
\begin{remark}\label{rem:multi1}
Indecomposable solutions of multipermutational level 1 are solutions of the form $C(1,n,0)$, for some $n$. Note that they always have a cyclic permutation group.
\end{remark}
\begin{example}\label{exm:1}
We describe some solutions  of the form $\mathcal{C}(n_1,n_2,r)$ for chosen $n_1,n_2$.
\begin{enumerate}
\item Let $n_1=1$, $n_2\in\Z^+$ and consider the solution $(X,\sigma,\tau)=\mathcal{C}(1,n_2,r)$. In this case we can omit the
first coordinate of the cartesian product $\Z_1\times \Z_{n_2}$ and we obtain
$\sigma_{i}(j)=j+ir+1$ and $\tau_{j}(i)=i-(j+1)r-1$. We are interested now in two special cases:
\begin{enumerate}
\item for $n_2=pq$, where $p$ and $q$ are two distinct prime numbers, $r$ must be equal $0$. Hence, for the solution
$\mathcal{C}(1,pq,0)$ we have  $\sigma_{i}(j)=j+1$ and $\tau_{j}(i)=i-1$.
\item if $n_2=p^2$, for a prime $p$, then $r=pt$ for some $t\in\{0,\ldots,p-1\}$. Since $r^2\equiv 0\pmod {p^2}$, the solution $\mathcal{C}(1,p^2,pt)$ for each such $t$ is the following: $\sigma_{i}(j)=j+pti+1$ and $\tau_j(i)=i-1-pt(j+1)$.
\end{enumerate}
\item Finally, if we take $n_1=n_2$ then $r=0$. In this case we obtain $(X,\sigma,\tau)=\mathcal{C}(n_1,n_1,0)$
 with $\sigma_{(a,i)}((b,j))=(b+i,j+1)$ and $\tau_{(b,j)}((a,i))=(a-(j+1),i-1)$.
\end{enumerate}
All the described solutions for primes $p$ and $q$, namely $\mathcal{C}(1,pq,0), \mathcal{C}(1,p^2,pt)$, for $t\in\{0,\ldots,p-1\}$ with $pt<p^2-1$, and $\mathcal{C}(p,p,0)$
 are exactly the same ones as in \cite[Corollary 22]{CPR}.

\end{example}

It is useful to see that different choices of the
parameters of the construction yield non-isomorphic solutions. For this, we need a technical lemma.

\begin{lemma}\label{lem:1}
 $\sigma^{n_1}_{\sigma_{(a,i)}((a,i))}=\sigma_{(a,i)}^{(r+1)n_1}$,
 for each $(a,i)\in \Z_{n_1}\times \Z_{n_2}$.
\end{lemma}

\begin{proof}
 Let us denote by $\delta_{(a,i)}=i-ar$ to simplify the formula \eqref{eq:indec}:
 \[
 \sigma_{(a,i)}((b,j))=(b+\delta_{(a,i)},j+r\delta_{(a,i)}+1).
 \]
 In particular,
 \[
 \sigma_{(a,i)}((a,i))=(a+\delta_{(a,i)},i+r\delta_{(a,i)}+1),
 \]
 and
\begin{align*}
&\sigma_{\sigma_{(a,i)}((a,i))}((b,j))=\sigma_{(a+\delta_{(a,i)},i+r\delta_{(a,i)}+1)}((b,j))=\\
&(b-r(a+\delta_{(a,i)})+i+r\delta_{(a,i)}+1,j+r(i+r\delta_{(a,i)}+1)-r^2(a+\delta_{(a,i)})+1)=\\
&(b-ra+i+1,j+ri+r-r^2a+1)=(b+\delta_{(a,i)}+1,j+r(\delta_{(a,i)}+1)+1).
\end{align*}
Now notice that for any $k\in \Z^+$
 \[
 \sigma_{(a,i)}^k((b,j))=(b+k\delta_{(a,i)},j+k(r\delta_{(a,i)}+1)),
 \]
 and
 \[
\sigma^k_{\sigma_{(a,i)}((a,i))}((b,j))=(b+k(\delta_{(a,i)}+1),j+k(r(\delta_{(a,i)}+1)+1)).
 \]
 Hence, we obtain
  \begin{align*}
& \sigma_{(a,i)}^{(r+1)n_1}((b,j))=(b+(r+1)n_1\delta_{(a,i)},j+(r+1)n_1(r\delta_{(a,i)}+1))=\\
&(b,j+rn_1+rn_1\delta_{(a,i)}+n_1)=(b+n_1(\delta_{(a,i)}+1),j+n_1(r(\delta_{(a,i)}+1)+1))=\\
&\sigma_{\sigma_{(a,i)}((a,i))}^{n_1}((b,j)),
 \end{align*}
which finishes the proof.
\end{proof}

\begin{proposition}\label{prop:6}
Two finite solutions $(X,\sigma,\tau)=\mathcal{C}(n_1,n_2,r)$ and $(X',\sigma',\tau')=\mathcal{C}(m_1,m_2,s)$ are
isomorphic if and only if $n_1=m_1$, $n_2=m_2$ and $r=s$.
 \end{proposition}

\begin{proof}
 Suppose that $(X,\sigma,\tau)\cong (X',\sigma',\tau')$. Clearly,
 isomorphic solutions have isomorphic permutation groups.
Moreover, two finite abelian groups are isomorphic if and only if they have the same invariant factors.
Hence, $n_1=m_1$ and $n_2=m_2$.
We just need to prove that $r=s$.

 Let $\varphi$ be an isomorphism from $(X,\sigma,\tau)$ to $(X',\sigma',\tau')$.
 Choose $(a,i)\in X$ and let~$(a',i')=\varphi((a,i))$. Then, by Lemma \ref{lem:1},
 \[
  \sigma_{(a',i')}^{(s+1)n_1}=
  \sigma_{\sigma_{(a',i')}((a',i'))}^{n_1}=
  \sigma_{\sigma_{\varphi((a,i))}(\varphi((a,i)))}^{n_1}=
  \varphi\sigma_{\sigma_{(a,i)}((a,i))}^{n_1}\varphi^{-1}=\varphi\sigma_{(a,i)}^{(r+1)n_1}\varphi^{-1}=\sigma_{(a',i')}^{(r+1)n_1}
 \]
 and both permutations are equal, for all~$a,i$, if and only if $r\equiv s\pmod{n_2/n_1}$. Since $0\leq r,s< n_2/n_1$, this forces $r=s$.
\end{proof}

\section{Solutions with an abelian permutation group}\label{sec2b}

In this section we prove that
the construction from Theorem~\ref{thm:1}
is a generic one, that means,
this way we obtain all the finite solutions of multipermutational level at most 2 with an abelian permutation group.
 Let us start with some auxiliary observations.
 The first one is true for infitite solutions too.

\begin{proposition}\label{prop:3}
 Let $(X,\sigma,\tau)$ be an indecomposable
 solution of multipermutational level at most~$2$
 with an abelian permutation group. Then
 \begin{enumerate}
  \item the permutation group $\mathcal{G}(X)$ is generated by at most two elements,
  \item for all $x\in X$, $\sigma_x$ have the same order,
  \item for each~$x\in X$ and $i\in \Z^+$, $\sigma_{\sigma^i_x(x)}=\sigma_{\sigma_x(x)}^i\sigma_x^{i+1}$.
 \end{enumerate}
\end{proposition}

\begin{proof}
  Let $(X,\sigma,\tau)$ be an indecomposable
 solution of multipermutational level at most~$2$
 with an abelian permutation group. Choose~$e\in X$ and let $\rho=\sigma_e$.
 Let $(X,L,\mathbf{R})$ be the $\rho^{-1}$ isotope of~$(X,\sigma,\tau)$.
 By Theorem \ref{th:2persol} this solution is $2$-reductive.
 Therefore
 $\rho$ satisfies Condition \eqref{eq:sigma}, namely
 \[ L_{\rho(x)}\rho L_y=L_{\rho(y)}\rho L_x.\]
 Since all the permutations commute, we obtain
 \[ L_{\rho(x)} L_y=L_{\rho(y)}L_x\]
 and, by substituting $x=\rho^{i-1}(e)$ and $y=e$, we get
 \[ L_{\rho^i(e)} =L_{\rho(e)}L_{\rho^{i-1}(e)}.\]
 Therefore, by an induction on~$i$,
 upwards towards $\infty$ and downwards towards $-\infty$, we get
 \begin{align}\label{eq:rho}
  L_{\rho^i(e)} =L^i_{\rho(e)},
  \end{align}
 for all $i\in\Z$. Hence

 \[\sigma_{\rho^i(e)}=L_{\rho^i(e)}\rho=L_{\rho(e)}^i\rho=(\sigma_{\rho(e)}\sigma_e)^i\sigma_e,
\] proving (3).

 Consider now any $x\in X$.
 The permutation group $\mathcal{G}(X)$ is transitive and therefore
 there exist $k\in\Z$ and $y_1,\ldots,y_k\in X$, such that
  \[x=\sigma_{y_1}\sigma_{y_2}\cdots \sigma_{y_k}(e)=L_{y_1}L_{y_2}\cdots L_{y_k}\rho^k(e).\] Then, by $2$-reductivity
 \[ \sigma_x=L_x\rho = L_{L_{y_1}L_{y_2}\cdots L_{y_k}\rho^k(e)}\rho = L_{\rho^k(e)}\rho=L_{\rho(e)}^k\rho\]
 and therefore the permutation group is generated
 by $\rho=\sigma_e$ and $L_{\rho(e)}=\sigma_{\sigma_e(e)}\sigma_e^{-1}$,
 proving (1).

 Now, suppose $\rho^m=\mathrm{id}$, for some $m\in \Z^+$. Then we have
 \[ \sigma_x^m=(L_{\rho(e)}^k\rho)^m=L_{\rho^{km}(e)}\rho^m=L_e=\mathrm{id} \]
 and therefore $o(\sigma_x) \leq o(\sigma_e)$.
 However, the element~$e$ was chosen arbitrarily, hence
 we can interchange the role of~$x$ and~$e$, obtaining
 (2).
 \end{proof}
As we can see, by Proposition \ref{prop:3}, for each (finite) indecomposable
 solution of multipermutational level at most~$2$ with abelian permutation group -- this group is
  either cyclic or is generated exactly by two generators.

\begin{lemma}\label{lm:4}
 Let $n_1,n_2\in\Z^+$ and $G$ has invariant factors $(n_1,n_2)$.
 Let $a,b\in G$ be such that $o(a)=n_2$.
 Then $n_1\cdot b\in\langle n_1\cdot a\rangle$.
\end{lemma}

\begin{proof}
 Without loss of generality suppose $G= \Z_{n_1}\times \Z_{n_2}$.
 The claim is evident for $n_1=n_2$, suppose hence $n_2> n_1$. Denote $a=(i,j)$ and $b=(k,l)$, for some $i,j,k,l\in\Z$.  The order of~$a$ is~$n_2$ hence
 $j$ is coprime to~$n_2$ and therefore
 there exists $j^{'}\in \Z_{n_2}$ such that $jj^{'}\equiv 1\pmod{n_2}$.
 Hence, $n_1\cdot b=n_1\cdot (k,l)=(0,n_1l)=(0,n_1jj^{'}l)=j^{'}l\cdot(n_1i,n_1j) = n_1j^{'}l\cdot (i,j)=n_1j^{'}l\cdot a$.
\end{proof}

\begin{rem}
Lemma \ref{lm:4} is not true in the infinite case,
even under assumption that $G$ is generated
by~$a$ and $b$.

Suppose, e.g. that $G=\Z_3\times\Z$, $a=(1,3)$ and $b=(0,2)$.
Then $3\cdot b=(0,6)$ is not a~multiple of~$3\cdot a=(0,9)$.
\end{rem}

In Section \ref{ssec:4} we shall continue the discussion on infinite solutions but now we focus on finite ones: we prove that these solutions have exactly the same properties as we have seen in Theorem~\ref{thm:1} and Lemma~\ref{lem:1}.

\begin{theorem}\label{prop:5}
 Let $(X,\sigma,\tau)$ be a finite indecomposable
 solution of multipermutational level at most~$2$
 with an abelian permutation group. Denote $|X|=n$. Then
 \begin{enumerate}
  \item there exist unique numbers~$n_1,n_2\in \Z^+$ such that $n_1 \mid n_2$ and
  $n=n_1\cdot n_2$ and the permutation group  $\mathcal{G}(X)$ has invariant factors $(n_1,n_2)$,
  \item there exists a unique number~$r\in\{0,1,\ldots n_2/n_1-1\}$, such that
  \begin{itemize}
   \item $n_2\mid n_1r^2$,
   \item for each~$x\in X$, $\sigma_{\sigma_{x}(x)}^{n_1}=\sigma_x^{(r+1)n_1}$,
  \end{itemize}
  \item for each~$x\in X$, define $\rho=\sigma_x$
   and $\lambda=\sigma_{\rho(x)}\rho^{-1-r}$; then $o(\rho)=n_2$, $o(\lambda)=n_1$ and $\mathcal{G}(X)=\langle\lambda\rangle\times\langle\rho\rangle$,
  \item for each $x\in X$, $\sigma_{\lambda^i\rho^j(x)} = \lambda^{j-ir}\rho^{1+jr-ir^2}$.
 \end{enumerate}
\end{theorem}

\begin{proof}
 Let $(X,\sigma,\tau)$ be a finite indecomposable
 solution of multipermutational level at most~$2$
 with an abelian permutation group. According to Proposition~\ref{prop:3}, the permutation group $\mathcal{G}(X)$ is an abelian group generated by permutations of the same order.
 Denote the order by~$n_2$. Actually, the exponent of $\mathcal{G}(X)$ has to be $n_2$ since $\mathcal{G}(X)$ is a $\Z_{n_2}$-module.
 Now, $\mathcal{G}(X)$
 is an abelian group acting transitively on~$X$ and hence we have $|X|=|\mathcal{G}(X)|=n$. This implies that there is $n_1\in \Z^+$ with $n=n_1n_2$ and $n_1\mid n_2$
 and  (1) follows from the fundamental theorem of finite abelian groups (see e.g. \cite[Theorem 6.13]{R}).

 Let $e,x\in X$ and $L_x=\sigma_x\sigma^{-1}_e$. By the end of the proof of Proposition~\ref{prop:3},
 the permutation group $\mathcal{G}(X)$ is generated by (at most) two permutations: $\rho=\sigma_e$ and $L_{\rho(e)}=\sigma_{\sigma_e(e)}\sigma_e^{-1}$.
The permutation
 $\rho$ is an element of order~$n_2$ and therefore,
 according to Lemma~\ref{lm:4}, there exists $r\in \Z$
 such that $L_{\rho(e)}^{n_1}=\rho^{n_1r}$.
 Actually, since $o(\rho)=n_2$, we can take $0\leq r<n_2/n_1$.
 Moreover, by Theorem \ref{th:2persol} we have $L_{L_x(y)}=L_y$, for each $x,y\in X$.
Then, by \eqref{eq:rho}
 \[ \mathrm{id}=L_e=L_{L_{\rho(e)}^{n_1}(e)}=L_{\rho^{n_1r}(e)}=(L_{\rho(e)}^{n_1})^r=
  (\rho^{n_1r})^r=\rho^{n_1r^2}
 \]
 from which we obtain $n_2 \mid n_1r^2$.

Furthermore,
\[
\sigma_e^{n_1r}=\rho^{n_1r}=L_{\rho(e)}^{n_1}=\sigma_{\sigma_e(e)}^{n_1}\sigma_e^{-n_1},
\]
which implies that $\sigma_{\sigma_e(e)}^{n_1}=\sigma_e^{(r+1)n_1}$ proving uniqueness of such~$r$ and therefore (2).

It was shown in the proof of Proposition~\ref{prop:3} that  there exists $k\in \Z$ such that, for each $x\in X$,
$\sigma_x=L_{\rho(e)}^k\rho$.
Let $\lambda=L_{\rho(e)}\rho^{-r}$. Then we have
 \[
 \sigma_x=L_{\rho(e)}^k\rho=(L_{\rho(e)}\rho^{-r})^k\rho^{kr+1}=\lambda^l\rho^{kr+1},
 \]
 which means that  $\mathcal{G}(X)=\langle \rho,\lambda\rangle$ and $o(\lambda)\geq n_1$. Now
 \[ \lambda^{n_1}=(L_{\rho(e)}\rho^{-r})^{n_1}=
 L_{\rho(e)}^{n_1}\rho^{-n_1 r}
 =\rho^{n_1 r}\rho^{-n_1 r}=\mathrm{id}\]
 proving (3).

 Suppose now $x\in X$ is an arbitrary element. Then there exist $i,j\in \Z$ such that $x=\lambda^i\rho^j(e)$.
 By Theorem \ref{th:2persol} and \eqref{eq:rho} we obtain
 \begin{align*}
  &\sigma_x=L_{\lambda^i\rho^j(e)}\rho=
  L_{(L_{\rho(e)}\rho^{-r})^i\rho^j(e)}\rho=L_{\rho^{j-ir}(e)}\rho=
  L_{\rho(e)}^{j-ir}\rho\\
  &
  =(\lambda\rho^{r})^{j-ir}\rho
  =\lambda^{j-ir}\rho^{1+r(j-ir)}=\lambda^{j-ir}\rho^{1+jr-ir^2}
 \end{align*}
which completes the proof.
\end{proof}

Now we are close to the
main result of the paper.
\begin{mthm}\label{cor:sol}
Each finite indecomposable solution of multipermutational level at most~$2$
 with abelian permutation group is isomorphic to $\mathcal{C}(n_1,n_2,r)$, for some $n_1,n_2\in\Z^+$ and $0\leq r<n_2/n_1$.
\end{mthm}

\begin{proof}
Consider a finite indecomposable solution~$(X,\sigma,\tau)$ of multipermutational level at most~$2$
and with abelian permutation group. According to Theorem~\ref{prop:5}, there are coefficients $n_1$, $n_2$ and~$r$ satisfying conditions of Theorem~\ref{thm:1}.
We shall prove $(X,\sigma,\tau)\cong \mathcal{C}(n_1,n_2,r)$.

Let us choose $e\in X$ and denote $\rho=\sigma_e$ and $\lambda=\sigma_{\rho(e)}\rho^{-r-1}$.
We define $\varphi:\mathcal{C}(n_1,n_2,r)\to (X,\sigma,\tau)$ as follows:
\[\varphi ((a,i)) = \lambda^a\rho^i(e).\]
The mapping is a well defined bijection, according to Theorem~\ref{prop:5}~(3).
 Then by Theorem \ref{prop:5}~(4)
 \begin{multline*}
  \varphi(\sigma_{(a,i)}(b,j))=
  \varphi(b-ar+i,j+ir-ar^2+1)
  =\lambda^{b-ar+i}\rho^{j+ir-ar^2+1}(e)=\\
  =\lambda^{-ar+i}\rho^{ir-ar^2+1}\lambda^b\rho^j(e)
  =\sigma_{\lambda^a\rho^i(e)} \lambda^b\rho^j(e)
  =\sigma_{\varphi((a,i))}\varphi((b,j))
  \end{multline*}
 and therefore $\varphi$ is an isomorphism.
\end{proof}

\begin{theorem}\label{thm:7}
 Let $n\in\Z^+$ and let $k$ be the largest integer, such that $k^2$ divides~$n$.
 Then there exist
  $\sum_{d \mid k}\limits \dfrac{k}{d}$
 indecomposable solutions with $n$ elements and of multipermutational level at most~$2$ with its permutation group abelian.
\end{theorem}

\begin{proof}
By Theorem \ref{cor:sol}, each finite indecomposable solution of multipermutational level at most~$2$
 with abelian permutation group is isomorphic to $\mathcal{C}(n_1,n_2,r)$, for some $n_1,n_2,r$ specified in Theorem \ref{thm:1}. Then
 it remains to count how many solutions are there,
 for each factorization $n=n_1\cdot n_2$
 with $n_1\mid n_2$.

 Let $l=n/k$.
 Clearly $n_1\mid k\mid l\mid n_2$.
 Suppose that $n$, $k$ and $l$ are fixed and we shall count the number of possible configurations $n_1$, $n_2$ and $r$.
 Whenever $n_1$ is a divisor of~$k$,
 and $0\leq r< n_2/n_1$,
 we have
 \[
n_2\mid n_1r^2\quad \Leftrightarrow\quad n\mid (n_1r)^2\quad \Leftrightarrow\quad l^2\mid (n_1r)^2\quad \Leftrightarrow\quad
l\mid n_1r\quad \Leftrightarrow\quad l/n_1 \mid r,
  \]
where we used that $l^2$ is the smallest square which is a multiple of~$n$.
Hence, suitable choices for $r$ are elements of the set $\{0,l/n_1,2l/n_1,\cdots,(n_2-l)/n_1\}$. The cardinality
 of this set is $n_2/l$ which is equal to $k/n_1$.
\end{proof}

Recall, by Proposition \ref{prop:3}, an abelian permutation group of a finite indecomposable
 solution $(X,\sigma,\tau)$ of multipermutational level at most~$2$,
 is generated by at most two elements. In the case when it is generated by only one element, i.e. the group is cyclic, by Theorem \ref{prop:5} we immediately obtain that the solution is isomorphic to $\mathcal{C}(1,n,r)$, for arbitrary $n\in \Z^+$ and some $r\in \{0,1,\ldots,n-1\}$ and, for all $x\in \Z_1\times \Z_n$,
 \[\sigma_{\sigma_{x}^i(x)}=\sigma_x^{ri+1}.\]

The next corollary directly follows by Proposition \ref{prop:6}.

\begin{cor}\label{prop:10}
 Let $(X,\sigma,\tau)=\mathcal{C}(1,n,r)$ and $(X',\sigma',\tau')=\mathcal{C}(1,m,s)$ be finite indecomposable
 solutions of multipermutational level at most~$2$ with cyclic permutation groups.
 Then they are isomorphic if and only if $n=m$
 and $r=s$.
\end{cor}

The analogue of Theorem \ref{thm:7} for solutions with cyclic permutation group is the following.

\begin{cor}\label{thm:11}
 Let $n\in\mathbb{N}$. The quantity of all indecomposable
 solutions of size $n$ of multipermutational level at most~$2$ with cyclic permutation group is equal to the largest number~$k$, such that $k^2$ divides~$n$.
\end{cor}
\begin{proof}
 It is sufficient to prove that the size of the set
 $ \{r\ |\ 0\leq r<n \ \& \ n \mid r^2\}$ is equal to~$k$.
 Let $n=k\cdot l$. Then $k$ divides~$l$ and therefore
 $n$ divides $(i\cdot l)^2$, for each $0\leq i< k$, which completes the proof.
\end{proof}
\begin{example}\cite[Theorem 16]{CPR}
Let $p$ be a prime number. By Corollary \ref{thm:11}, there are $p$ indecomposable solutions of cardinality $p^2$ of multipermutational level at most $2$ with cyclic permutation group.
\end{example}
\begin{example}
Let $p$ be a prime number. By Theorem \ref{thm:7} we have
\[
\sum_{d \mid p^2}\limits \dfrac{p^2}{d}=p^2+p+1
\]
indecomposable solutions with $p^4$ elements of multipermutational level at most~$2$ and with abelian permutation group; among them, exactly $p^2$ have a cyclic permutation group,
by Corollary \ref{thm:11}.

 For instance, for $p=2$, there are 7 such solutions:
\begin{itemize}
\item For $n_1=1$ and $n_2=16$, the possible choices for $r$ are: $0$, $4$, $8$ or $12$. Hence, we have four solutions:
\begin{itemize}
\item [$\bullet$] $\mathcal{C}(1,16,0)=(\Z_{16},\sigma,\tau)$, with $\sigma_{i}(j)=j+1$;
\item [$\bullet$] $\mathcal{C}(1,16,4)=(\Z_{16},\sigma,\tau)$, with $\sigma_{i}(j)=j+4i+1$;
\item [$\bullet$] $\mathcal{C}(1,16,8)=(\Z_{16},\sigma,\tau)$, with $\sigma_{i}(j)=j+8i+1$;
\item [$\bullet$] $\mathcal{C}(1,16,12)=(\Z_{16},\sigma,\tau)$, with $\sigma_{i}(j)=j+12i+1$.
\end{itemize}
In this case, all the permutation groups $\mathcal{G}(\Z_{16})$ are cyclic.
\vskip 1mm
\item For $n_1=2$ and $n_2=8$, possible choices for $r$ are $0$ or $2$. Hence, we have two solutions:
\begin{itemize}
\item [$\bullet$] $\mathcal{C}(2,8,0)=(\Z_{2}\times \Z_8,\sigma,\tau)$, with $\sigma_{(a,i)}((b,j))=(b+i,j+1)$;
\item [$\bullet$] $\mathcal{C}(2,8,2)=(\Z_{2}\times \Z_8,\sigma,\tau)$, with $\sigma_{(a,i)}((b,j))=(b-2a+i,j+2i-4a+1)$.
\end{itemize}
\vskip 1mm
\item For $n_1=n_2=4$, the only possible choice for $r$ is $0$. Hence, we have one solution:
\begin{itemize}
\item [$\bullet$] $\mathcal{C}(4,4,0)=(\Z_{4}\times \Z_4,\sigma,\tau)$, with $\sigma_{(a,i)}((b,j))=(b+i,j+1)$.
\end{itemize}
\end{itemize}
\end{example}

\section{Automorphism group}\label{ssec3}

In this section we compute some properties of the automorphism group of the constructed solutions
and we give some examples.

\begin{proposition}\label{prop:2}
 Let $(X,\sigma,\tau)=\mathcal{C}(n_1,n_2,r)$, for some $n_1,n_2,r\in \Z^+$, be a finite indecomposable
 solution of multipermutational level at most~$2$ with abelian permutation group. Then
 every endomorphism of the solution ~$(X,\sigma,\tau)$ is an automorphism and
 the automorphism group $\aut{X}$ of~$(X,\sigma,\tau)$ is a regular abelian group.
\end{proposition}

\begin{proof}
By the proof of Theorem \ref{thm:1}, for $(a,i), (b,j)\in \Z_{n_1}\times \Z_{n_2}$,
\[\sigma_{(a,i)}((b,j))=(b-ar+i,j+ir-ar^2+1)=(b+\delta_{(a,i)},j+r\delta_{(a,i)}+1),\]
where $\delta_{(a,i)}=i-ar$.

Notice that
 \[
 \sigma_{(a,i)}^k((b,j))=(b+k\delta_{(a,i)},j+k(r\delta_{(a,i)}+1))
 \]
and
\[
 \sigma_{(0,1)}^a\sigma_{(0,0)}^{\delta_{(a,i)}-a} ((b,j))=(b+a,j+i).\]

 Let~$f$ be an endomorphism of~$(X,\sigma,\tau)$.
 Denote $f((0,0))=(s,t)$. We prove that~$f$ is already determined by~$s$ and~$t$. Indeed,
 \[ f((0,1))=f(\sigma_{(0,0)}(0,0))
  =\sigma_{(s,t)}(s,t)=(s+\delta_{(s,t)},t+r\delta_{(s,t)}+1).
 \]
For each $(a,i)\in \Z_{n_1}\times \Z_{n_2}$,
 \begin{align*}
  f((a,i)) &= f(\sigma_{(0,1)}^a\sigma^{\delta_{(a,i)}-a}_{(0,0)}((0,0)))
  =\sigma_{(s+\delta_{(s,t)},t+r\delta_{(s,t)}+1)}^a\sigma^{\delta_{(a,i)}-a}_{(s,t)}((s,t))\\
  &=\sigma_{(s+\delta_{(s,t)},t+r\delta_{(s,t)}+1)}^a (s+(\delta_{(a,i)}-a)\cdot \delta_{(s,t)},t+(\delta_{(a,i)}-a)\cdot(r\delta_{(s,t)}+1))\\
  &=(s+(\delta_{(a,i)}-a)\cdot \delta_{(s,t)}+a\cdot(t+r\delta_{(s,t)}+1-(s+\delta_{(s,t)})r),\\
  &\strut\kern2cm t+(\delta_{(a,i)}-a)\cdot(r\delta_{(s,t)}+1)+a\cdot((t+r\delta_{(s,t)}+1)r-(s+\delta_{(s,t)})r^2+1))\\
  &=(s+(\delta_{(a,i)}-a)\cdot \delta_{(s,t)}+a\cdot(\delta_{(s,t)}+1),
  t+(\delta_{(a,i)}-a)\cdot(r\delta_{(s,t)}+1)+a\cdot(r\delta_{(s,t)}+r+1))\\
  &=(s+a+\delta_{(a,i)}\delta_{(s,t)},t+i+r\cdot\delta_{(a,i)}\delta_{(s,t)}).
 \end{align*}

Let now $f_{(s,t)}$ be the mapping that
sends $(a,i)$ to $(s+a+\delta_{(a,i)}\delta_{(s,t)},t+i+r\cdot\delta_{(a,i)}\delta_{(s,t)})$.
We shall prove that $f_{(s,t)}$ is a~homomorphism.
\begin{align*}
 f_{(s,t)} (\sigma_{(a,i)} ((b,j))) &=
 f_{(s,t)} ((b+\delta_{(a,i)},j+r\delta_{(a,i)}+1))\\
 &=(s+b+\delta_{(a,i)}+(j+r\delta_{(a,i)}+1-r(b+\delta_{(a,i)}))\delta_{(s,t)},\\
 &\strut\kern2cm t+j+r\delta_{(a,i)}+1+r(j+r\delta_{(a,i)}+1-r(b+\delta_{(a,i)}))\delta_{(s,t)})\\
 &=(s+b+\delta_{(a,i)}+(\delta_{(b,j)}+1)\delta_{(s,t)}, t+j+r\delta_{(a,i)}+1+r(\delta_{(b,j)}+1)\delta_{(s,t)})\\
 \sigma_{f_{(s,t)}((a,i))} (f_{(s,t)}((b,j))) &=
 \sigma_{(s+a+\delta_{(a,i)}\delta_{(s,t)},t+i+r\cdot\delta_{(a,i)}\delta_{(s,t)})}
 ( (s+b+\delta_{(b,j)}\delta_{(s,t)},t+j+r\cdot\delta_{(b,j)}\delta_{(s,t)}) )\\
 &= (s+b+\delta_{(b,j)}\delta_{(s,t)}+t+i+r\cdot\delta_{(a,i)}\delta_{(s,t)}-r(s+a+\delta_{(a,i)}\delta_{(s,t)}),\\
 &\strut\kern12mm t+j+r\cdot\delta_{(b,j)}\delta_{(s,t)}+r(t+i+r\cdot\delta_{(a,i)}\delta_{(s,t)})) -r^2(s+a+\delta_{(a,i)}\delta_{(s,t)})+1)\\
 &= (s+b+\delta_{(b,j)}\delta_{(s,t)}+\delta_{(s,t)}+\delta_{(a,i)},
 t+j+r(\delta_{(b,j)}\delta_{(s,t)}+\delta_{(s,t)}+\delta_{(a,i)})+1).
\end{align*}

Then we notice that
\[\delta_{f_{(s,t)}((a,i))}=t+i+r\delta_{(a,i)}\delta_{(s,t)}-r\cdot(s+a+\delta_{(a,i)}\delta_{(s,t)})
 = \delta_{(s,t)}+\delta_{(a,i)}.
\]
Now we shall prove that all the endomorphisms
of~$(X,\sigma,\tau)$ commute:
\begin{multline*}
 f_{(s,t)}f_{(u,v)}((a,i))=
 f_{(s,t)} ((u+a+\delta_{(a,i)}\delta_{(u,v)},v+i+r\delta_{(a,i)}\delta_{(u,v)})\\
 =(s+u+a+\delta_{(a,i)}\delta_{(u,v)}+\delta_{f_{(u,v)}((a,i))}\delta_{(s,t)},
  t+v+i+r\delta_{(a,i)}\delta_{(u,v)}+r\delta_{f_{(u,v)}((a,i))}\delta_{(s,t)})\\
 =(s+u+a+\delta_{(a,i)}\delta_{(u,v)}+\delta_{(u,v)}\delta_{(s,t)}+\delta_{(a,i)}\delta_{(s,t)},
  t+v+i+r(\delta_{(a,i)}\delta_{(u,v)}+\delta_{(u,v)}\delta_{(s,t)}+\delta_{(a,i)}\delta_{(s,t)}))
\end{multline*}
and this expression is symmetric with respect to the exchange $(s,t)\leftrightarrow (u,v)$.

Finally, we observe that $f_{(s,t)}^{-1}=f_{(\delta_{(s,t)}^2-s,r\delta_{(s,t)}^2-t)}$. Indeed,
\[ \delta_{(\delta_{(s,t)}^2-s,r\delta_{(s,t)}^2-t)}=\delta_{(-s,-t)}=-\delta_{(s,t)}\]
and the rest follows by substituting $(u,v)$
by $(\delta_{(s,t)}^2-s,r\delta_{(s,t)}^2-t)$
in the previous computation.
Hence $f_{(s,t)}$ is an automorphism of~$(X,\sigma,\tau)$ and
$\aut{X}$ is an abelian transitive group.

\end{proof}

Computing the exact structure of the automorphism group is more complicated and therefore we do not do it here.
We just want to give some examples to show that this group need not be isomorphic to the permutation group.

\begin{example}
Let $p$ be a prime number. By Theorem \ref{thm:7} there are $p+1$
indecomposable solutions with $p^2$ elements of multipermutational level at most~$2$ and with abelian permutation group. If $p=2$ we have three such solutions:
\begin{itemize}
\item [$\bullet$] $\mathcal{C}(1,4,0)=(\Z_{4},\sigma,\tau)$, with $\sigma_{i}(j)=j+1$;
\item [$\bullet$] $\mathcal{C}(1,4,2)=(\Z_{4},\sigma,\tau)$, with $\sigma_{i}(j)=j+2i+1$;
\item [$\bullet$] $\mathcal{C}(2,2,0)=(\Z_{2}\times \Z_2,\sigma,\tau)$, with $\sigma_{(a,i)}((b,j))=(b+i,j+1)$.
\end{itemize}

For the first solution $(\Z_{4},\sigma,\tau)=\mathcal{C}(1,4,0)$, $\sigma_{0}=\sigma_{1}=\sigma_2=\sigma_3=(0123)$ and the permutation group $\mathcal{G}(\Z_{4})=\langle(0123)\rangle\cong\Z_4$. In this case the automorphism group $\aut{\Z_{4}}$ is the same.

For the second  solution $(\Z_{4},\sigma,\tau)=\mathcal{C}(1,4,2)$, we have $\sigma_{0}=\sigma_{2}=(0123)$ and $\sigma_1=\sigma_3=(3210)$ and also $\mathcal{G}(\Z_{4})=\langle(0123)\rangle\cong\Z_4$. But now the automorphism group
 is
 \[\aut{\Z_{4}}=\{\mathrm{id},(01)(23),(02)(13),(03)(12)\}\cong\Z_2\times \Z_2.\]

Finally, for the third  solution $(\Z_{2}\times \Z_2,\sigma,\tau)=\mathcal{C}(2,2,0)$:
$\sigma_{(0,0)}=\sigma_{(1,0)}=((0,0)\ (0,1))((1,0)\ (1,1))$, $\sigma_{(0,1)}=\sigma_{(1,1)}=((0,0)\ (1,1))((0,1)\ (1,0))$ and $\mathcal{G}(\Z_{2}\times \Z_2)\cong\Z_2\times \Z_2$.
 The automorphism group is
 \[
 \aut{\Z_{2}\times \Z_2}=\{\mathrm{id},((0,0)\ (0,1)\ (1,0)\ (1,1)), ((0,0)\ (1,0))((0,1)\ (1,1)), ((0,0)\ (1,1)\ (1,0)\ (0,1))\}\cong\Z_4.\]
\end{example}

Computations of $\aut{X}$ simplify in the case of cyclic permutation groups.

\begin{proposition}\label{prop:12}
Let $(X,\sigma,\tau)=\mathcal{C}(1,n,r)$ be a finite indecomposable
 solution of multipermutational level at most~$2$ with a cyclic permutation group.
The automorphism group of~$(X,\sigma,\tau)$ is not cyclic if and only if $n\equiv 0\pmod 4$ and $r\equiv 2\pmod 4$.
\end{proposition}

\begin{proof}
By Proposition~\ref{prop:2}, the group $\aut{X}$ is regular and abelian.
Let us investigate whether the automorphism group
is cyclic.

Let $f_k$, for $k\in\Z_n$, be an automorphism of $\mathcal{C}(1,n,r)$ with $f_k(0)=k$. In Proposition \ref{prop:2} we have seen that
 $f_k(i)=k+i\cdot(rk+1)$.
 We now prove, by an induction on~$j$, that
 $f^j_k(i)=j(j-1) rk^2/2+jk(ir+1)+i$. The claim is true
 for~$j=0$. Now
 \begin{multline*}
 f^j_k(i)=f_k(f_k^{j-1}(i))=f_k((j-1)(j-2)rk^2/2+(j-1)k(ir+1)+i)=\\
 k+\big((j-1)(j-2)rk^2/2+(j-1)k(ir+1)+i\big)(rk+1)=\\
 k+(j-1)rk^2+rki+(j-1)(j-2)rk^2/2+(j-1)k(ir+1)+i=j(j-1)rk^2/2+jk(ir+1)+i.
 \end{multline*}
 Thus
 $f^r_k(i)=(r-1)r^2k^2/2+rk+i$
 and, if $r$ is odd, we have
 $f^r_k(i)=rk+i$. In this case, by taking, for instance,
 $k=1$, we see that $f_1^r$ is an automorphism of
 order~$n/\mathrm{gcd}(n,r)$ and therefore
 $f_1$ is of order~$n$. The same situation happens
 if~$r$ is even and~$n$ is odd as we can take $f^{r+n}$ instead of~$f^r$.

A more complicated situation occurs when
 $r$ is even and~$n$ is even as well.
 If~$k$ is even, we have $f_k^{n/2}(i)=(n/2)k+i=i$.
 If~$k$ is odd then
 $f^{n/2}_k(i)=((n/2-1)rk/2+1)kn/2+i$.
 If~$r$ is divisible by~$4$ then we have $f^{n/2}_k(i)=n/2+i$ and $f_k$ is an automorphism of order~$n$. The only remaining situation is thus $n$ even, $k$ even and $r\equiv 2\pmod 4$.

 Suppose that $n$ is divisible by~$4$. Then $(n/2-1)$
 is odd and $((n/2-1)\cdot k\cdot r/2+1)$ is even.
 We have thus $f_k^{n/2}(i)=i$ and no automorphism is of order~$n$. On the other hand if $n\equiv 2\pmod 4$ then $(n/2-1)$ is even and $((n/2-1)\cdot k\cdot r/2+1)$ is odd. Therefore $f_k^{n/2}(i)=n/2+i$ and $f_k$ is an automorphism of order~$n$.
\end{proof}

\section{Infinite solutions}\label{ssec:4}
At the end we shall discuss some infinite solutions.
First of all we observe that the proofs of Theorem~\ref{thm:1} and Proposition~\ref{prop:2} go through even in the case $n_2=\infty$ and $r=0$.

\begin{theorem}\label{thm:8}
 Let $n_1\in \Z^+\cup\{\infty\}$. Then
 the set~$X=\Z_{n_1}\times \Z$ with
 $\sigma_{(a,i)}((b,j))=(b+i,j+1)$
 is an infinite indecomposable solution
 of multipermutational level at most~$2$
 with the multiplication group isomorphic to $\Z_{n_1}\times \Z$. Its automorphism group is a regular abelian group.
\end{theorem}

\begin{proof}
 Let us go through the proofs of Theorem~\ref{thm:1} and Proposition~\ref{prop:2}.
 Finiteness was assumed in three places only:
  the first one was the begining of the proof of Theorem~\ref{thm:1} where we checked that the expression
 is well defined. This is settled here by setting $r=0$.

 The second place was the check that the cycle set is non-degenerate. Every finite cycle set is non-degenerate, hence the finite case is granted for free. In the infinite case we have to check that the mapping $T:X\to X$ defined as
 \[ T((a,i))=\sigma^{-1}_{(a,i)}((a,i))=(a-i,i-1)\]
 is a bijection. It is easy to see that $T^{-1}((a,i))=(a+i+1,i+1)$.

 The third finiteness argument is implicit in the proof of Proposition~\ref{prop:2},
 by saying that transitivity implies regularity.
 But here $r=0$ and therefore $f_{(s,t)}((a,i))=(s+a+it,t+i)$ and we see that $f_{(0,0)}$ is the only automorphism with fixed points.
\end{proof}

Unfortunately, we do not know whether the construction of Theorem~\ref{thm:8}
is the only possible one since we do not have
an equivalent of Proposition~\ref{prop:5} in the infinite case. In order to obtain some, we should get past
the obstacle of Lemma~\ref{lm:4} not working in for infinite groups.

Nevertheless, if we assume that the permutation group is cyclic then such a solution is uniquely determined.

\begin{proposition}\label{prop:13}
 There exists a unique infinite indecomposable solution of multipermutational level at most 2 with a cyclic permutation group.
 Its group of automorphisms is equal to its permutation group.
\end{proposition}

\begin{proof}
Let $(X,\sigma,\tau)$ be an infinite indecomposable solution of multipermutational level at most 2 with cyclic permutation group 
and let~$e\in X$.
According to Proposition~\ref{prop:3}~(1),
the infinite group is generated by $\sigma_e$ and $\sigma_{\sigma_e(e)}$, hence, according to
Proposition~\ref{prop:3}~(2)
the order of~$\sigma_e$ is infinite.
Denote $\rho=\sigma_e$ and let  $(X,L,\mathbf{R})$ be the $\rho^{-1}$-isotope of~$(X,\sigma,\tau)$.
Denote $\lambda=L_{\rho(e)}$ and we know that $\mathcal{G}(X)$ is generated by~$\rho$ and $\lambda$.

The permutation group of~$(X,\sigma,\tau)$ is cyclic
and therefore there exist exponents $l$ and $m$ such that $\lambda^l\rho^m=\gamma$ is a generator of the group.
Furthemore, there exists~$r\in\Z$ such that
 $\lambda=L_{\rho(e)}=\gamma^r$. Then, by $2$-reductivity of  $(X,L,\mathbf{R})$,
 \[ \mathrm{id}=L_e=L_{L_{\rho(e)}(e)}=L_{\gamma^r(e)}=L_{(\lambda^l\rho^m)^r(e)}=
 L_{\rho^{mr}(e)}=L^{mr}_{\rho(e)}=\gamma^{mr^2}\]
and therefore either $m=0$ or $r=0$.

Let now $x\in X$. Since $(X,\sigma,\tau)$ is indecomposable and $\mathcal{G}(X)$ is cyclic, there exists~$k\in \Z$, such that
$x=\gamma^k(e)$. Now
\[ L_x=L_{\gamma^k(e)}=L_{\lambda^{lk}\rho^{mk}(e)}=L_{L_{\rho(e)}^{lk}\rho^{mk}(e)}
=L_{\rho^{mk}(e)}
=L_{\rho(e)}^{mk}=\lambda^{mk}=\gamma^{rmk}=\mathrm{id}.\]
 
Hence $\sigma_x=L_x\rho=\rho$, for all $x\in  X$.
 Thus the solution is a~permutation solution.

 Now, let~$f$ be an automorphism of $(X,\sigma,\tau)$. We have
 \[f(\sigma_x(y))=f\rho(y) \quad \text{ and }\quad
  \sigma_{(f(x))} (f(y)) = \rho f(y)
 \]
and therefore $f$ is a power of~$\rho$
and $\aut{X}\subseteq \mathcal{G}(X)$. The other inclusion is evident.
\end{proof}

\end{document}